\documentclass[11pt,a4paper]{article}

\usepackage[main=british, german, french]{babel}
\usepackage[babel=true]{csquotes}

\usepackage{pgf,tikz}
\usepackage{physics}
\usepackage{mathrsfs}
\usetikzlibrary{arrows}
\usetikzlibrary{arrows.meta}
\usepackage[T1]{fontenc}
\usepackage{amsmath}
\usepackage{amsfonts}
\usepackage{amssymb}
\usepackage{amsthm}
\usepackage{amscd}
\usepackage{graphicx}
\usepackage{enumitem}
\usepackage{tikz}
\usepackage{ifthen}
\usepackage{verbatim}

\usepackage{hyperref}
\usepackage{amsmath,caption}
\usepackage{amsfonts}
\usepackage{amssymb}
\usepackage{amsthm}

\usepackage{url,pdfpages,xcolor,framed,color}
\usepackage{a4wide}

\newtheorem{thr}{Theorem}[section]
\newtheorem{q}[thr]{Question}
\newtheorem{lemma}[thr]{Lemma}
\newtheorem{proposition}[thr]{Proposition}
\newtheorem{conj}[thr]{Conjecture}
\def\zz{\mathbb{Z}}

\defineshorthand{"`}{\openautoquote}
\defineshorthand{"'}{\closeautoquote}

\title{Punctured intervals tile $\zz^3$}
\author{Stijn Cambie\footnote{Department of Mathematics, Radboud University Nijmegen, Postbus 9010, 6500 GL Nijmegen, The Netherlands. Email: \href{mailto:S.Cambie@math.ru.nl}{S.Cambie@math.ru.nl}. This work has been supported by a Vidi Grant of the Netherlands Organization for Scientific Research (NWO), grant number $639.032.614$.} }%
\date{}

\begin{document}
	\definecolor{xdxdff}{rgb}{0.49019607843137253,0.49019607843137253,1.}
	\definecolor{ududff}{rgb}{0.30196078431372547,0.30196078431372547,1.}

	\maketitle

	\begin{abstract}
		Extending the methods of Metrebian (2018), we prove that punctured intervals tile $\zz^3$. This solves two questions of Metrebian and completely resolves a question of Gruslys, Leader and Tan. 
		We also pose a question that asks whether there is a relation between the genus $g$ (number of holes) in a one-dimensional tile $T$ and a uniform bound $d$ such that $T$ tiles $\zz^d$. 
		An affirmative answer would generalize a conjecture of Gruslys, Leader and Tan (2016). 
\end{abstract}

\section{Introduction}

	Given $n$, let $T$ be a tile in $\zz^n$, i.e.~a finite subset of $\zz^n$.
	The cardinality of $T$, $\lvert T \rvert$, is the size of $T$, i.e. the number of elements of the subset.
	Confirming a conjecture of Chalcraft that was posed on MathOverflow, Gruslys, Leader and Tan~\cite{GLT} showed that $T$ 
	tiles $\zz^d$ for some $d$.
	This is an existence result and they wondered about better bounds in terms of the dimension $n$ and the size $\lvert T \rvert.$
	They conjectured the following for the case $n=1$.
	
	\begin{conj}[Gruslys, Leader, Tan~\cite{GLT}] \label{GLT}
		For any positive integer $t$ there is a number $d$ such that any tile $T$ in $\zz$
		with $\lvert T \rvert =t$ tiles $\zz^d$.
	\end{conj}
	
	Let us note that Adler and Holroyd~\cite{AH} had earlier investigated which tiles in $\zz$ can tile $\zz$.
	When dealing with one-dimensional tiles, we find it convenient to use the same notation as in~\cite{AH}: a tile $T$ in $\zz$ which is the union of $n$ intervals $I_1$ up to $I_n$, such that the length of interval $I_i$ is $a_i$ and the gap between $I_i$ and $I_{i+1}$ is $b_i$, will be denoted by $a_1(b_1)a_2(b_2)a_3 \ldots (b_{n-1})a_n$.
	We will call $g=n-1$ the genus of the interval $T$.
	
	Wondering about Conjecture~\ref{GLT}, one may wonder if the dimension $d$ only depends on the genus of the tile instead of the size. Leading to the following question.
	
	\begin{q}\label{q:genus_tiling}
		Does there exist a function $d \colon \mathbb N \to \mathbb N$ such that any tile $T \subset \zz$ with genus $g$ tiles $\zz^{d(g)}?$
	\end{q}
	Answering this affirmatively, would confirm Conjecture~\ref{GLT} since $g \le t-1.$ As observed in e.g.~\cite{GLT}, for any fixed $d$, there are one-dimensional tiles with large genus which cannot tile $\zz^d$, see Section~\ref{1a}. In particular we observe that $d(g)\ge \frac g2 +1$ by taking $k \to \infty$ in Proposition~\ref{prop:TkgDn}.
	We make progress on Question~\ref{q:genus_tiling} for the case $g=1$ by establishing that one-dimensional tiles $T=k(m)\ell$ with $m \le 2$ do tile $\zz^3$, as sketched in Appendix~\ref{5}.
	In full detail, we prove that punctured intervals do tile $\zz^3$ as our main result. 
	\begin{thr}\label{main}
		Every punctured interval $T=k(1)\ell$ does tile $\zz^3.$
	\end{thr}
	
	This theorem answers two concrete questions posed by Metrebian~\cite[Qu.~10,11]{HM}.
	As a corollary the least $d$ for which $T=k(1)k$ tiles $\zz^d$ equals $\min \{k,3\},$ answering~\cite[Qu.~21]{GLT}.

	In Section~\ref{2} we prove a lemma implying that it is enough to find some structured partial tilings of $\zz^2$ to prove tiles do tile $\zz^3.$ 
	In Section~\ref{3}, we exhibit such partial tilings for punctured intervals. 
	Some divisibility constraints for specific constructions make this a delicate task. For the symmetric tiles $T=k(1)k$ the construction depends on $v_2(k)$, the exponent of $2$ in the prime factorization of k. So we create infinitely many families of constructions. This is done in Lemma~\ref{v_2_case}. Metrebian~\cite{HM} did have such examples already when $v_2(k)\in \{0,2\}.$ 
	An important ingredient to prove the validity of this infinite family of partial tilings is some elementary number theory.
	In Lemma~\ref{k1lcase} we give constructions for asymmetric tiles $T=k(1)\ell$ where $k \not= \ell.$ So Lemma~\ref{v_2_case} and Lemma~\ref{k1lcase} together imply Theorem~\ref{main}.

\section{From partial to complete tilings}\label{2}


In this section, we will prove that finding certain partial tilings is enough to conclude that a whole tiling does exist.
This is done in Lemma~\ref{lemma4gen} which is a generalization of Lemma~$4$ in~\cite{HM}.

\begin{lemma}\label{lemma4gen}
Let $T$ be the one-dimensional tile $k(m)\ell$. 
Suppose there are three disjoint subsets $A, B, C$ of $\zz^d$ with the same cardinality such that one can tile $\zz^d \setminus (A \cup B)$, $\zz^d \setminus (A \cup C)$ and $\zz^d \setminus (B \cup C)$ with $T$. Then $T$ tiles $\zz^{d+1}$.
\end{lemma}

\begin{proof}
First assume $m< \min\{k,\ell\}.$
We construct a subset $Y \subset \zz \times \{0,1,2\}$ such that $\lvert Y \cap \left( \{z\} \times \{0,1,2\} \right) \rvert =2$ for every $z \in \zz$ and such that $T$ tiles $Y.$
Let $(x,i) \in Y$ for some $x \in \zz$ and $i \in \{0,1,2\}$ if and only if
\begin{align*}
x-i(k+l) &\equiv 1,2,\ldots, k; k+m+1,k+m+2 \ldots, k+m+\ell \pmod{3k+3\ell} \mbox{ or}\\
 &\equiv 2k+\ell+1, \ldots, 2k+2\ell;2k+2\ell+m+1,\ldots,3k+2\ell+m \pmod{3k+3\ell}.
\end{align*}

The construction has been sketched in Figure~\ref{fig:VertCon} for $\{1,2,\ldots, 3(k+\ell)\}\times \{0,1,2\} $. By gluing infinitely many copies of that picture together, one gets the full construction of $Y$.

\begin{figure}[h]
	\centering
	\begin{tikzpicture}[scale=0.82]

	\draw[fill=black!50!white] \foreach \x in {0,11} {
		\foreach \a in {0} { 
			
			(\x,\a) rectangle (\x+2,\a+0.5)
		}
	};
	
	\draw[fill=black!50!white] \foreach \x in {3,7} {
		\foreach \a in {0} { 
			
			(\x,\a) rectangle (\x+3,\a+0.5)
		}
	};
	
	\draw[fill=black!50!white] \foreach \x in {1,5} {
		\foreach \a in {1} { 
			
			(\x,\a) rectangle (\x+2,\a+0.5)
		}
	};
	
	\draw[fill=black!50!white] \foreach \x in {8,12} {
		\foreach \a in {1} { 
			
			(\x,\a) rectangle (\x+3,\a+0.5)
		}
	};
	
	\draw[fill=black!50!white] \foreach \x in {10,6} {
		\foreach \a in {2} { 
			
			(\x,\a) rectangle (\x+2,\a+0.5)
		}
	};
	
	\draw[fill=black!70!white] \foreach \x in {13} {
		\foreach \a in {2} { 
			
			(\x,\a) rectangle (\x+2,\a+0.5)
		}
	};
	
	\draw[fill=black!50!white] \foreach \x in {2} {
		\foreach \a in {2} { 
			
			(\x,\a) rectangle (\x+3,\a+0.5)
		}
	};
	
	\draw[fill=black!70!white] \foreach \x in {0} {
		\foreach \a in {2} { 
			
			(\x,\a) rectangle (\x+1,\a+0.5)
		}
	};

	\draw[dotted] (13.25,-0.5)--(13.75,0)--(13.75,3.25)--(13.25,2.75)--(13.25,-0.5);
	\draw[dotted] (12.25,-0.5)--(12.75,0)--(12.75,3.25)--(12.25,2.75)--(12.25,-0.5);
	\draw[dotted] (11.25,-0.5)--(11.75,0)--(11.75,3.25)--(11.25,2.75)--(11.25,-0.5);

	\draw[<->] (0.05,0.65) -- (1.95,0.65){};
	\draw[<->] (2.05,0.65) -- (2.95,0.65){};
	\draw[<->] (3.05,0.65) -- (5.95,0.65){};

	\draw[<->] (6.05,0.65) -- (6.95,0.65){};
	\draw[<->] (7.05,0.65) -- (9.95,0.65){};
	\draw[<->] (10.05,0.65) -- (10.95,0.65){};
	\draw[<->] (11.05,0.65) -- (12.95,0.65){};
	
	\draw[<->] (0.05,1.65) -- (4.95,1.65){};
	\draw[<->] (0.05,2.65) -- (9.95,2.65){};
	\coordinate [label=center:$k$] (A) at (1,0.8); 
	\coordinate [label=center:$m$] (A) at (2.5,0.8); 
	\coordinate [label=center:$\ell$] (A) at (4.5,0.8); 
	
	\coordinate [label=center:\small{$k-m$}] (A) at (6.5,0.8); 
	\coordinate [label=center:$m$] (A) at (10.5,0.8); 
	\coordinate [label=center:$\ell$] (A) at (8.5,0.8); 
	\coordinate [label=center:$k$] (A) at (12,0.8); 
	
	\coordinate [label=center: \small{$k+\ell$}] (A) at (2.5,1.8); 
	\coordinate [label=center:$2(k+\ell)$] (A) at (5,2.8); 	
	\coordinate [label=center:$0$] (A) at (-0.5,0.25); 
	\coordinate [label=center:$1$] (A) at (-0.5,1.25); 
	\coordinate [label=center:$2$] (A) at (-0.5,2.25); 
	
	\coordinate [label=center:$\pi_1$] (A) at (11,-0.75); 
	\coordinate [label=center:$\pi_2$] (A) at (12.5,-0.75); 
	\coordinate [label=center:$\pi_3$] (A) at (13.5,-0.75); 
	\end{tikzpicture}
	\caption{ Construction of $Y$.}
	\label{fig:VertCon}
	
\end{figure}

Now we explain why this construction meets the conditions we need.
Let $S_1= \{1,2,\ldots, k\}$, $S_2= \{k+m+1,k+m+2 \ldots, k+m+\ell\}$, $S_3=\{2k+\ell+1, \ldots, 2k+2\ell\}$ and $S_4=\{2k+2\ell+m+1,\ldots,3k+2\ell+m\}$.
Let $S_o=S_1+S_3$ and $S_e=S_2+S_4$.
Then both $S_o \cup \left( (k+\ell)+S_o \right) \cup \left( 2(k+\ell) + S_o \right)$ and $S_e \cup \left( (k+\ell)+S_e \right) \cup \left(2(k+\ell) + S_e\right)$ cover all elements in $\frac{\zz}{3(k+\ell)\zz}$ exactly once, from which the result follows.

The elements of $A \cup B \cup C$ can be partitioned into triples $\{a_i, b_i, c_i\}$ since $A, B,C$ have the same cardinality.
Every set $\zz \times \{a_i, b_i, c_i\} $ has a subset $Y_i \cong Y$ which can be tiled by $T$ in the same manner, i.e.~there exists a partition $\{Z_1, Z_2, Z_3\}$ of $\zz$ such that for every $i$ we have $Y_i \cap \left( \{z\} \times  \{a_i, b_i, c_i\}\right)= \{a_i, b_i\}$ for every $z \in Z_1$, $Y_i \cap \left(\{z\} \times  \{a_i, b_i, c_i\}\right)= \{a_i, c_i\}$ for every $z \in Z_2$ and $Y_i \cap \left( \{z\} \times  \{a_i, b_i, c_i\}\right)= \{ b_i, c_i\}$ for every $z \in Z_3.$
Now $\zz^{d+1} \setminus \left(\cup_i Y_i \right)$ can be written as $ Z_1 \times \left( \zz^d \setminus (A \cup B) \right)  \cup Z_2 \times \left( \zz^d \setminus (A \cup C) \right) \cup Z_3 \times \left( \zz^d \setminus (B \cup C) \right) $ 
 and by the assumptions this can be tiled by $T$ as well, so $T$ tiles $\zz^{d+1}$.
Looking at Figure~\ref{fig:VertCon}, every hyperplane $\pi_i$ will be covered by the intersections with $\cup_i Y_i$ and a partial tiling isomorphic to one of $\zz^d \setminus (A \cup B),  \zz^d \setminus (A \cup C)$ or $\zz^d \setminus (B \cup C).$

When $m \ge \min\{k,\ell\}$, where we assume without loss of generality $k=\min\{k,\ell\},$ one can glue two copies $T_1, T_2$ of $T$ together to a tile $T'$ with $k'=\ell'=k+\ell$ and $m'=m-k$ by taking $T_1=\{-k,-k+1, \ldots, -1\} \cup \{m,m+1,\ldots m+\ell-1\}$ and $T_2=\{-k-\ell,-k-\ell+1,\ldots,-k-1\} \cup \{ m-k,m-k+1, \ldots,m-1\}$. See Figure~\ref{Reduce} for a depiction.
When $m' \ge k'$, one can glue $\left \lfloor m'/k' +1 \right \rfloor$ copies of $T'$ together, which are translates of $T'$ with initial point at $0,k', \ldots, \left \lfloor m'/k'  \right \rfloor k'$.
Hence we have reduced this to the case which has been proven already.%
\begin{figure}[h]
	\begin{center}
	\begin{tikzpicture}[scale=0.82]
	\draw[fill=black!60!white] 			(0,0) rectangle (2.5,0.6);
	\draw[fill=black!60!white] 			(10,0) rectangle (11.5,0.6);
	\draw[fill=black!40!white] 			(2.5,0) rectangle (4,0.6);
	\draw[fill=black!40!white] 			(11.5,0) rectangle (14,0.6);	
	\draw[fill=black!20!white] 			(4,0) rectangle (8,0.6);
\draw[fill=black!20!white] 			(14,0) rectangle (18,0.6);
	\draw[<->] (0.05,0.75) -- (2.45,0.75){};
	\draw[<->] (2.55,0.75) --(9.95,0.75){};
	\draw[<->] (10.05,0.75) -- (11.45,0.75){};
	\draw[<->] (10.05,-0.4) -- (13.95,-0.4){};
	\draw[<->] (0.05,-0.4) -- (3.95,-0.4){};
	\draw[<->] (4.05,-0.4) -- (9.95,-0.4){};
	\coordinate [label=center:$\ell$] (A) at (1.25,0.95); 
	
	\coordinate [label=center:$m$] (A) at (6,0.95); 
	\coordinate [label=center:$k$] (A) at (10.75,0.95); 
	\coordinate [label=center: $\ell'$=$k+\ell$] (A) at (12,-0.2); 
	\coordinate [label=center: $k'$=$k+\ell$] (A) at (2,-0.2); 
	\coordinate [label=center:$m'$=$m-k$] (A) at (7,-0.2); 	
		\coordinate [label=center: $T_2$] (A) at (1.25,0.25); 
	\coordinate [label=center:$T_1$] (A) at (3.25,0.25); 	
	\coordinate [label=center:$T'_2$] (A) at (6,0.25); 
	\end{tikzpicture}
\end{center}
	\caption{ Gluing $T_1$ and $T_2$ and copies $T'$.}
	\label{Reduce}
\end{figure}
\end{proof}

\section{Punctured intervals tile $\zz^3$ }\label{3}

Throughout this section, we let $T$ be a punctured interval tile, which is the union of an interval of length $k$ and an interval of length $\ell$ with a gap of size $1$. So $T=k(1)\ell$ equals a translate of $\{-k,-k+1,\ldots, -1,1,2,\ldots,\ell\}$ as a subset of $\zz.$ 
By applying Lemma~\ref{lemma4gen}, we will prove that $T$ tiles $\zz^3$ for any $k,\ell.$

As a warm up and for completeness of presentation we construct three partial tilings of the plane satisfying the conditions of Lemma~\ref{lemma4gen} when $T$ is the symmetric punctured interval $k(1)k$ with $k \equiv 1 \pmod 2$. This was also proven in \cite[Thr.~3]{HM}.

\begin{figure}[h]
	\begin{center}
		\begin{tikzpicture}[scale=0.75]
		
		\clip (-0.5, -0.5) rectangle (6.05,6.05);

\foreach \t in {-2,-1,0,1,2,3,4,5,6} { 
	\foreach \a in {-6,-4,-2,0,2,4,6} { 
		\ifthenelse{\a>-\t}{\draw[fill=black!70!white](\t,\t+\a) rectangle (\t+0.5,\a+0.5+\t);}{;}}
};

\foreach \t in {-2,-1,0,1,2,3,4,5,6} { 
	\foreach \a in {-6,-4,-2,0,2,4,6} {  
		\ifthenelse{\a>-\t}{\draw[fill=black!40!white](\t+0.5,\t+\a+0.5) rectangle (\t+1,\a+1+\t);}{;}}
	};

\foreach \t in {-4,-3,-2,-1,0,1,2,3,4,5,6} { 
	\foreach \a in {-6,-4,-2,0,2,4,6} {
		\ifthenelse{\a>-\t}{\draw[fill=black!10!white](\t+0.5,\t+\a) rectangle (\t+1,\a+0.5+\t);}{;}}
};

\draw[fill=black!70!white] 			(1.5,-0.5) rectangle (2,0);
\draw[fill=black!40!white] 			(3,-0.5) rectangle (3.5,0);
\draw[fill=black!10!white] 			(4.5,-0.5) rectangle (5,0);

		\coordinate [label=center:legend:] (A) at (0.5,-0.25); 
\coordinate [label=center:$A$] (A) at (2.25,-0.25); 
\coordinate [label=center:$B$] (A) at (3.75,-0.25); 
\coordinate [label=center:$C$] (A) at (5.25,-0.25);

		\draw[<->] (1.25,1.5) -- (1.25,3){};
		\draw[<->] (0.05,0.85) -- (1.55,0.85){};
		\coordinate [label=center:$k$] (A) at (1.5,2.25); 
		\coordinate [label=center:$k$] (A) at (0.8,1.15);
		\end{tikzpicture}
	\end{center}
	\caption{ Construction of $A,B,C$ for $T=k(1)k$ where $k \equiv 1 \pmod 2$.}
	\label{ABC_1mod2}	
\end{figure}

\begin{figure}[h]
	\begin{center}
		\begin{tikzpicture}[scale=0.75]
		
		\clip (-0.5, -0.5) rectangle (5,6);

		\foreach \t in {-2,-1,0,1,2,3,4,5,6} { 
			\foreach \a in {-6,-4,-2,0,2,4,6} {  
			\draw[fill=black!40!white](\t+0.5,\t+\a+0.5) rectangle (\t+1,\a+1+\t);
				}
		};

		\foreach \t in {-4,-3,-2,-1,0,1,2,3,4,5,6} { 
			\foreach \a in {-6,-4,-2,0,2,4,6} {
				\draw[fill=black!10!white](\t+0.5,\t+\a) rectangle (\t+1,\a+0.5+\t);
			}
		};
	
		\foreach \t in {0,1} { 
		\foreach \a in {-6,-4,-2,0,2,4,6} {
			\draw[fill=red](\t+1,\t+\a) rectangle (\t+2.5,\a+0.5+\t);
			\draw[fill=red](\t+3,\t+\a) rectangle (\t+4.5,\a+0.5+\t);
		}
	};

	\foreach \t in {-2,-1} { 
	\foreach \a in {-6,-4,-2,0,2,4,6} {
		\draw[fill=yellow](\t+1,\t+\a) rectangle (\t+2.5,\a+0.5+\t);
		\draw[fill=yellow](\t+3,\t+\a) rectangle (\t+4.5,\a+0.5+\t);
	}
};

	\foreach \t in {-2,-1} { 
	\foreach \a in {-6,-4,-2,0,2,4,6} {
		\draw[fill=blue](\t+1,\t+\a+0.5) rectangle (\t+2.5,\a+1+\t);
		\draw[fill=blue](\t+3,\t+\a+0.5) rectangle (\t+4.5,\a+1+\t);
	}
};

		\foreach \t in {0,1} { 
	\foreach \a in {-6,-4,-2,0,2,4,6} {
		\draw[fill=green](\t+1,\t+\a+0.5) rectangle (\t+2.5,\a+1+\t);
\draw[fill=green](\t+3,\t+\a+0.5) rectangle (\t+4.5,\a+1+\t);
	}
};

		\end{tikzpicture}
		\quad
		\begin{tikzpicture}[scale=0.75]
		
		\clip (-0.5, -0.5) rectangle (5,6);

		\foreach \t in {-2,-1,0,1,2,3,4,5,6} { 
			\foreach \a in {-6,-4,-2,0,2,4,6} { 
				\draw[fill=black!70!white](\t,\t+\a) rectangle (\t+0.5,\a+0.5+\t);
			}
		};

		\foreach \t in {-2,-1,0,1,2,3,4,5,6} { 
			\foreach \a in {-6,-4,-2,0,2,4,6} {  
				\draw[fill=black!40!white](\t+0.5,\t+\a+0.5) rectangle (\t+1,\a+1+\t);
			}
		};
		%
		

				\foreach \t in {0,1} { 
			\foreach \a in {-6,-4,-2,0,2,4,6} {
				\draw[fill=red](\t+4.5,\t+\a) rectangle (\t+6,\a+0.5+\t);
				\draw[fill=red](\t+2.5,\t+\a) rectangle (\t+4,\a+0.5+\t);
			}
		};
		
		\foreach \t in {-2,-1} { 
			\foreach \a in {-6,-4,-2,0,2,4,6} {
				\draw[fill=yellow](\t+0.5,\t+\a) rectangle (\t+2,\a+0.5+\t);
				\draw[fill=yellow](\t+2.5,\t+\a) rectangle (\t+4,\a+0.5+\t);
			}
		};
		
		\foreach \t in {-2,-1} { 
			\foreach \a in {-6,-4,-2,0,2,4,6} {
				\draw[fill=blue](\t+1,\t+\a+0.5) rectangle (\t+2.5,\a+1+\t);
				\draw[fill=blue](\t+3,\t+\a+0.5) rectangle (\t+4.5,\a+1+\t);
			}
		};
		
		\foreach \t in {0,1} { 
			\foreach \a in {-6,-4,-2,0,2,4,6} {
				\draw[fill=green](\t+1,\t+\a+0.5) rectangle (\t+2.5,\a+1+\t);
				\draw[fill=green](\t+3,\t+\a+0.5) rectangle (\t+4.5,\a+1+\t);
			}
		};

		\end{tikzpicture}
		\quad\begin{tikzpicture}[scale=0.75]
		
		\clip (-0.5, -0.5) rectangle (5,6);

		\foreach \t in {-2,-1,0,1,2,3,4,5,6} { 
			\foreach \a in {-6,-4,-2,0,2,4,6} { 
				\draw[fill=black!70!white](\t,\t+\a) rectangle (\t+0.5,\a+0.5+\t);
			}
		};

		\foreach \t in {-4,-3,-2,-1,0,1,2,3,4,5,6} { 
			\foreach \a in {-6,-4,-2,0,2,4,6} {
				\draw[fill=black!10!white](\t+0.5,\t+\a) rectangle (\t+1,\a+0.5+\t);
			}
		};

\foreach \t in {-4,-3,-2,-1,0,1,2,3,4,5,6} { 
		\foreach \a in {0,-8} {
			\draw[fill=green](\t+0.5,\t+\a+0.5) rectangle (\t+1,\t+\a+2);
			\draw[fill=green](\t+0.5,\t+\a+2.5) rectangle (\t+1,\a+4+\t);
		}
	};

\foreach \t in {-4,-3,-2,-1,0,1,2,3,4,5,6} { 
	\foreach \a in {-4,4} {
		\draw[fill=blue](\t+0.5,\t+\a+0.5) rectangle (\t+1,\t+\a+2);
		\draw[fill=blue](\t+0.5,\t+\a+2.5) rectangle (\t+1,\a+4+\t);
	}
};

\foreach \t in {-4,-3,-2,-1,0,1,2,3,4,5,6} { 
	\foreach \a in {0,-8} {
		\draw[fill=red](\t,\t+\a+0.5) rectangle (\t+0.5,\t+\a+2);
		\draw[fill=red](\t,\t+\a+2.5) rectangle (\t+0.5,\a+4+\t);
	}
};

\foreach \t in {-4,-3,-2,-1,0,1,2,3,4,5,6} { 
	\foreach \a in {-4,4} {
		\draw[fill=yellow](\t,\t+\a+0.5) rectangle (\t+0.5,\t+\a+2);
		\draw[fill=yellow](\t,\t+\a+2.5) rectangle (\t+0.5,\a+4+\t);
	}
};

		\end{tikzpicture}		
	\end{center}
	\caption{ Partial tilings for $k \equiv 1 \pmod 2$.}
	\label{ABC_1mod2conf}	
\end{figure}

\begin{lemma}[Metrebian~\cite{HM}]\label{easycase1}
	If $2 \nmid k$, then $T=k(1)k$ tiles $\zz^3.$
\end{lemma}

\begin{proof}
	Let $A=\{(x,y) \in \zz^2 \mid  x \equiv y \pmod{k+1}, x \equiv 0 \pmod 2\}$, $B=(1,1)+A$ and $C=(1,0)+A$.
	Now one can tile $\zz^2 \setminus (A \cup B)$ horizontally or vertically, as $A \cup B$ consists of infinitely many diagonals with distance $k$ in between.
	
	One can tile $\zz^2 \setminus (A \cup C)$ by placing copies of $T$ vertical and similarly one can tile $\zz^2 \setminus (B \cup C)$ with horizontal copies of $T$. 
	The sets $A,B,C$ are presented in Figure~\ref{ABC_1mod2}, with the three partial tilings being sketched in Figure~\ref{ABC_1mod2conf}.	
	Hence the result follows from Lemma~\ref{lemma4gen}.	
\end{proof}

\begin{lemma}\label{v_2_case}
	Symmetric punctured intervals $T=k(1)k$ tile $\zz^3.$
\end{lemma}

\begin{proof}
	Let $v_2(k)=n$ and $q=2^n.$ 
	When $n=0$, the result follows from Lemma~\ref{easycase1}. So from now on, we assume $n \ge 1$.
%

Let $A \subset \zz^2$ be the sets containing the elements $(x,y)$ if and only if
$$x-y\left[ 2(k+1)(q-1)+1 \right] \equiv i(k+1) \pmod{4(k+1)q} $$
for some $0 \le i \le 2q-1$.
Let $B=(2q(k+1),0)+A$ and $C=B+(k,0).$

One can see a depiction of this in Figure~\ref{fig:v2k=v2l} in the case $q=2, n=1$ (for $k=6$ actually).

\begin{figure}[h]
\centering
	\begin{tikzpicture}[scale=0.5]		
		
	\clip (0, -1.75) rectangle (28.55,17.05);
	\foreach \a in {0,1,2,3,4,5,6,7,8,9,10,11,12,13,14,15,16,17,18,19,20,21,22,23,24,25,26,27,28,29,30,31,32,34,33,35,36,37,38,39,40,41,42,43,44,45,46,47,48} { 
		\foreach \t in {-308,-168,-196,-280,-252,-224,-112,-140,-84,-56,-28,0,28} { 
			\foreach \x in {0,1,2,3} { 
			\draw[fill=black!10!white](\t+3.5*\x+7.5*\a,0.5*\a) rectangle (\t+3.5*\x+0.5+7.5*\a,0.5*\a+0.5);}
	};
};

	\foreach \a in {0,1,2,3,4,5,6,7,8,9,10,11,12,13,14,15,16,17,18,19,20,21,22,23,24,25,26,27,28,29,30,31,32,34,33,35,36,37,38,39,40,41,42,43,44,45,46,47,48} { 
	\foreach \t in {-308,-168,-196,-280,-252,-224,-112,-140,-84,-56,-28,0,28} {  
		\foreach \x in {0,1,2,3} { 
			\draw[fill=black!40!white](14+\t+3.5*\x+7.5*\a,0.5*\a) rectangle (14+\t+3.5*\x+0.5+7.5*\a,0.5*\a+0.5);}
	};
};

	\foreach \a in {0,1,2,3,4,5,6,7,8,9,10,11,12,13,14,15,16,17,18,19,20,21,22,23,24,25,26,27,28,29,30,31,32,34,33,35,36,37,38,39,40,41,42,43,44,45,46,47,48} { 
	\foreach \t in {-308,-168,-196,-280,-252,-224,-112,-140,-84,-56,-28,0,28} {  
		\foreach \x in {0,1,2,3} { 
			\draw[fill=black!70!white](17+\t+3.5*\x+7.5*\a,0.5*\a) rectangle (17+\t+3.5*\x+0.5+7.5*\a,0.5*\a+0.5);}
	};
};

	\draw[fill=black!10!white] 			(9,-1.5) rectangle (9.5,-1);

\draw[fill=black!40!white] 			(12,-1.5) rectangle (12.5,-1);

\draw[fill=black!70!white] 			(15,-1.5) rectangle (15.5,-1);
	
	\coordinate [label=center:legend:] (A) at (7,-1.25); 
	\coordinate [label=center:$A$] (A) at (10.25,-1.25); 
	\coordinate [label=center:$B$] (A) at (13.25,-1.25); 
	\coordinate [label=center:$C$] (A) at (16.25,-1.25);

	\end{tikzpicture}

	\caption{Construction of partial planar tilings where $v_2(k)=1$}
	\label{fig:v2k=v2l}

\end{figure}

Now one can tile $\zz^2 \setminus (A \cup B)$ with $T$ as $A \cup B$ is the union of diagonals which are distance $k+1$ apart.
One can tile $ \zz^2 \setminus (A \cup C)$ horizontally as well.
For this, it is enough to tile one horizontal line as every horizontal line is a translate of that one and due to periodicity in particular the set 
\begin{align*}
&\left( \zz^2 \setminus (A \cup C) \right) \cap \left( \{0,1,\ldots, 4q(k+1)-1\} \times \{0\} \right) \\&=\left( \{0,1,\ldots, 4q(k+1)-1\} \times \{0\} \right) \setminus (A \cup C).
\end{align*}
For this, use translates of $T$ starting at $\left( 1+2i(k+1),0\right) $ for $0 \le i \le q-1$ and at $\left( 2i(k+1),0 \right)$ for $q \le i \le 2q-1.$

To finish, we note that we can tile $\zz^2 \setminus (B \cup C)$ vertically.
For this, we only have to check $ \left(\{0\} \times \zz\right) \setminus (B \cup C),$ since $\gcd \{4(k+1)q,2(k+1)(q-1)+1\}=1$ and hence every vertical line is up to some translation identical to every other vertical line.
By noting that $B$ and $C$ are subsets of some diagonals on the plane, one checks that 
$$\left(\{0\} \times \zz \right) \cap B=\{0\} \times \{y\mid y \equiv i(k+1)\pmod{4(k+1)q}, 2q \le i \le 4q-1\}.$$ 
For this, note that 
\begin{align*}
&0-i(k+1)\cdot [2(k+1)(q-1)+1] \equiv   i(k+1) \pmod{4q(k+1)}
\end{align*}
since $k \equiv q \pmod{2q}.$
Similarly one has
$$\left(\{0\} \times \zz \right) \cap C=\{0\} \times \{y\mid y \equiv i(k+1)+1 \pmod{4(k+1))q}, -1 \le i \le 2(q-1)\}.$$ 

\noindent
Hence one can tile $\left(\{0\} \times \zz\right) \setminus (B \cup C)$ by putting vertical tiles starting at 
$(0,i(2k+2)-k+1)$ for every $i \equiv 0,1,\ldots, q-1 \pmod{2q}$ and 
$(0,i(2k+2)-k)$ for every $i \equiv q,q+1,\ldots, 2q-1 \pmod{2q}$.
Hence the result follows from Lemma~\ref{lemma4gen}.
\end{proof}

\begin{lemma}\label{k1lcase}
	Asymmetric punctured intervals $T=k(1)\ell$ with $k>\ell$ tile $\zz^3.$
\end{lemma}

\begin{proof}
	Let $A=\{(x,y)\mid y\equiv x-k \pmod{k+\ell+1}, x\equiv 1,2,\ldots, k-\ell \pmod{2(k-\ell)} \}$, $B=(k-\ell,k-\ell)+A$ and $C=(k-\ell,0)+A.$

	Note that $A \cup B$ form diagonals distance $k+\ell+1$ apart and so we can tile $\zz^2 \backslash (A \cup B)$ by putting all tiles horizontal or vertical.
	One can tile $\zz^2 \backslash (A \cup C)$ vertically with copies of $T$ and	$\zz^2 \backslash (B \cup C)$ horizontally.	
	A sketch of an example is given in Figure~\ref{fig:k(1)k-1}.
\end{proof}

\begin{figure}[h]
	\begin{center}
		\begin{tikzpicture}[scale=0.75]
		
		\clip (-0.5, -0.5) rectangle (6.05,6.05);

		\foreach \t in {-2,-1,0,1,2,3,4,5,6} { 
			\foreach \a in {-6,-2,2,6} { 
				\ifthenelse{\a>-\t}{\draw[fill=black!70!white](\t,\t+\a) rectangle (\t+0.5,\a+0.5+\t);}{;}}
		};

		\foreach \t in {-2,-1,0,1,2,3,4,5,6} { 
			\foreach \a in {-6,-2,2,6} { 
				\ifthenelse{\a>-\t}{\draw[fill=black!40!white](\t+0.5,\t+\a+0.5) rectangle (\t+1,\a+1+\t);}{;}}
		};

		\foreach \t in {-4,-3,-2,-1,0,1,2,3,4,5,6} { 
			\foreach \a in {-6,-2,2,6} { 
				\ifthenelse{\a>-\t}{\draw[fill=black!10!white](\t+0.5,\t+\a) rectangle (\t+1,\a+0.5+\t);}{;}}
		};

		\draw[fill=black!70!white] 			(1.5,-0.5) rectangle (2,0);
		\draw[fill=black!40!white] 			(3,-0.5) rectangle (3.5,0);
		\draw[fill=black!10!white] 			(4.5,-0.5) rectangle (5,0);
		
		\coordinate [label=center:legend:] (A) at (0.5,-0.25); 
		\coordinate [label=center:$A$] (A) at (2.25,-0.25); 
		\coordinate [label=center:$B$] (A) at (3.75,-0.25); 
		\coordinate [label=center:$C$] (A) at (5.25,-0.25);

		\draw[<->] (0.0,0.85) -- (3.5,0.85){};

		\coordinate [label=center:$k+\ell$] (A) at (1.8,1.15);
		\end{tikzpicture}
		\quad
				\begin{tikzpicture}[scale=0.75]
				
				\clip (-0.5, -0.5) rectangle (6.05,6.05);

				\foreach \t in {-4,-3,-2,-1,0,1,2,3,4,5,6} { 
					\foreach \a in {-6,-2,2,6} { 
						\ifthenelse{\a>-\t}{\draw[fill=black!40!white](\t+0.5,\t+\a+0.5) rectangle (\t+1,\a+1+\t);}{;}}
				};

				\foreach \t in {-4,-3,-2,-1,0,1,2,3,4,5,6} { 
					\foreach \a in {-6,-2,2,6} { 
						\ifthenelse{\a>-\t}{\draw[fill=black!10!white](\t+0.5,\t+\a) rectangle (\t+1,\a+0.5+\t);}{;}}
				};

				\foreach \t in{-11,-10,-9,-8,-7,-6,-5,-4,-3,-2,-1,0,1,2,3,4,5,6,7,8,9,10,11} { 
					\foreach \a in {-2,6,14} { 
						\ifthenelse{\a>-\t}{\draw[fill=blue](\t-1.5,\t+\a) rectangle (\t+0.5,\a+0.5+\t);}{;}
						\ifthenelse{\a>-\t}{\draw[fill=blue](\t+1,\t+\a) rectangle (\t+2.5,\a+0.5+\t);}{;}
					}
				};
			
			\foreach \t in {-11,-10,-9,-8,-7,-6,-5,-4,-3,-2,-1,0,1,2,3,4,5,6,7,8,9,10,11} { 
				\foreach \a in {-6,2,10} { 
					\ifthenelse{\a>-\t}{\draw[fill=red](\t-1.5,\t+\a) rectangle (\t+0.5,\a+0.5+\t);}{;}
					\ifthenelse{\a>-\t}{\draw[fill=red](\t+1,\t+\a) rectangle (\t+2.5,\a+0.5+\t);}{;}
				}
			};

\foreach \t in{-11,-10,-9,-8,-7,-6,-5,-4,-3,-2,-1,0,1,2,3,4,5,6,7,8,9,10,11} { 
	\foreach \a in {-2,6,14} { 
		\ifthenelse{\a>-\t}{\draw[fill=yellow](\t-1.5,\t+\a+0.5) rectangle (\t+0.5,\a+1+\t);}{;}
\ifthenelse{\a>-\t}{\draw[fill=yellow](\t+1,\t+\a+0.5) rectangle (\t+2.5,\a+1+\t);}{;}
	}
};

\foreach \t in {-11,-10,-9,-8,-7,-6,-5,-4,-3,-2,-1,0,1,2,3,4,5,6,7,8,9,10,11} { 
	\foreach \a in {-6,2,10} { 
		\ifthenelse{\a>-\t}{\draw[fill=green](\t-1.5,\t+\a+0.5) rectangle (\t+0.5,\a+1+\t);}{;}
		\ifthenelse{\a>-\t}{\draw[fill=green](\t+1,\t+\a+0.5) rectangle (\t+2.5,\a+1+\t);}{;}
	}
};

				\end{tikzpicture}

	\end{center}
	\caption{ Construction of $A,B,C$ for $T=k(1)\ell$ where $k =\ell+ 1$ with a partial tiling for $\zz^2 \backslash (B \cup C)$.}
	\label{fig:k(1)k-1}	
\end{figure}

\section{Impossible tilings}\label{1a}

In this section, for the convenience of the reader, we collect two classes of (known) one-dimensional tiles of high genus that do not tile $\zz^d$ for a given $d$.

Let $T_{k,g}$ be the tile $k\underbrace{(k-1)1(k-1)1 \ldots (k-1)1(k-1)}_{g \mbox{ times } (k-1)}k$.
Note that $T_{k,k}$ was considered in~\cite{GLT}.

Let $D_n$ be the tile $\underbrace{2(1)2(1)2\ldots(1)2}_{n \mbox{ times } 2}$, 
as considered in~\cite{K}.

The following proposition shows that for every $d$, one can find $\{k,g\}$ and $n$ such that neither $D_n$ nor $T_{k,g}$ tiles $\zz^d$. The reason behind this is slightly different for the two tiles.
The first uses sparseness of tiles put in one direction.
The other considers the intersection of the tiles with subdivisions of $\zz^d$.

\begin{proposition}\label{prop:TkgDn}
	$T_{k,g}$ does not tile $\zz^d$ for $d< \frac{kg+2k-1}{2k+g-1}$ and $D_n$ does not tile $\zz^d$ for $n>3^{d-1}$.
\end{proposition}

\begin{proof}
	In the case of $T_{k,g}$, one looks to the maximum volume covered by tiles in one of the $d$ orthogonal directions in a hypercube $[N]^d.$
	When $N \to \infty$, the ratio of the volume covered by these tiles will have a limsup which is at most $\frac{2k+g-1}{kg+2k-1}$ from which the result follows as the sum of the ratios over the $d$ directions should sum to $1$.\\
	Next, we consider the $D_n.$
	We assume $D_n$ tiles $\zz^d$ and look to the intersection of this fixed tiling with a hypercube $[N]^d$.
	Look to the $3^d$ possible (infinite) partitions of $\zz^d$ in hypercubes with side length $3$. These correspond with the $3^d$ possible vertices in $\left( \frac{\zz}{3\zz} \right)^d$ that represent the corners of all the cubes in the partition. Call a nonempty intersection of $[N]^d$ with a hypercube of side length $3$ for a given partition a subregion.
	We now count the total number $\#D$ of intersections of a subregion of a partition and a $D_n$ 
	which are of size $2$ (the intersection has not necessarily to be connected), in two different ways. 
	
	For each of the $3^d$ partitions, there are less than $\left( \frac{N}{3}+2 \right)^d$ subregions. Each subregion will contain at most $\frac{3^d-1}{2}$ intersections with a $D_n$ of size $2$.
	Hence $\#D < 3^d \left( \frac{N}{3}+2 \right)^d \frac{3^d-1}{2}= \frac{3^d-1}{2} (N+6)^d$.
	
	On the other hand, there are at least $\frac{(N-6n)^d}{2n}$ $D_n$'s completely inside the hypercube.
	Every $D_n$ of these, intersects $n$ subregions in exactly $2$ places for each of $2 \cdot 3^{d-1}$ partitions. For $3^{d-1}$ partitions, these $D_n$ intersects $n-1$ small hypercubes in exactly $2$ places and $2$ small hypercubes in exactly one place.
	This implies that $\# D \ge \frac{(N-6n)^d}{2n} \cdot 3^{d-1}(3n-1)$.
	
	Hence $\frac{3^d-1}{2} (N+6)^d >\frac{(N-6n)^d}{2n} \cdot 3^{d-1}(3n-1)$ for all $N$, in particular one finds that the leading coefficients satisfy 
	$\frac{3^d-1}{2} \ge 3^{d-1}\cdot \frac{3n-1}{2n} \Rightarrow n \le 3^{d-1}.$

\end{proof}

In the case of $D_n$, this generalizes the `only if' part of Proposition~$1$ in~\cite{K}.

Let us remark that this also follows from a straightforward generalization of Theorem~$1$ in~\cite{K}, which concerns `convolutions' of tiles. In case it might be of use to others, we use the notation of~\cite{K} to state the generalization (and leave the proof to the reader) of Theorem~$1$ in~\cite{K} (where they deal with the case $n=2$ and $d=2$).

\begin{proposition}[\cite{K}]\label{AUOK}
	Suppose $T \subset \zz^n$ is a tile.
	Suppose that $S \subset \zz^d$ is a symmetric tile (i.e.~no matter how the tile is oriented, it is a translate of itself).
	Then if for some $m \in \mathbb N$ one has $\lvert 1_S \bigstar_m 1_T \rvert_1 <\lvert \overline{1_S} \rvert \lvert 1_T \rvert$, or if $\lvert 1_S \bigstar_m 1_T \rvert_{\infty} < \lvert 1_T \rvert$ and $\lvert \overline{1_S} \rvert \not =0$, then $T$ does not tile $\zz^d$. 
\end{proposition}

\subsection*{Note}

This paper is an update of an older version entitled \enquote*{Symmetric punctured intervals tile $\zz^3$}. So we like to give an idea of the changes.
In this version, we start our journey more general from Conjecture~\ref{GLT} and Question~\ref{q:genus_tiling}. By finding a simple construction for the asymmetric punctured intervals in Lemma~\ref{k1lcase}, some results in the old versions could be stated more elegantly. Other results could be removed, as they became redundant. The main result now also holds for all punctured intervals and not only symmetric intervals.

\bibliographystyle{abbrv}
\bibliography{Tiling_1Dtiles_3D}

\begin{thebibliography}{1}

\bibitem{AH}
A.~Adler and F.~C. Holroyd.
\newblock Some results on one-dimensional tilings.
\newblock {\em Geom. Dedicata}, 10(1-4):49--58, 1981.

\bibitem{GLT}
V.~Gruslys, I.~Leader, and T.~S. Tan.
\newblock Tiling with arbitrary tiles.
\newblock {\em Proc. Lond. Math. Soc. (3)}, 112(6):1019--1039, 2016.

\bibitem{K}
A.~U.~O. Kisisel.
\newblock Polyomino convolutions and tiling problems.
\newblock {\em J. Combin. Theory Ser. A}, 95(2):373--380, 2001.

\bibitem{HM}
H.~Metrebian.
\newblock Tiling with punctured intervals.
\newblock {\em Mathematika}, 65(2):181--189, 2019.

\end{thebibliography}

\appendix
\section{The tiles $k(2)\ell$ do tile $\zz^3$ as well}\label{5}

In this section, we briefly sketch the construction for the cases where the gap is of length $2$.

If $2\mid k, \ell$, four copies form a tile that is isomorphic to $\frac{k}{2}(1)\frac{\ell}{2}$ and so it is known by Theorem~\ref{main}.
When $k\equiv \ell \pmod 4$ are both odd, we have $4 \mid k+\ell+2$ and a construction analogous to Figure~\ref{ABC_1mod2} does work ($k$ has to be $k+\ell$ instead and and every square being a $2 \times 2$ square).

The remaining cases where $k$ and $\ell$ are having different parity, or both are odd and $4\mid k+\ell$, are handled in the following proposition.

\begin{proposition}
	Let $k, \ell \in \zz^+$ and $d=\gcd\{k+1,\ell+1\}$ where $k+1=dx, \ell+1=dy$ and $x+y$ is odd.
	Then $k(2)\ell$ does tile $\zz^3.$
\end{proposition}

\begin{proof}
	We start with selecting $3$ transversals in $\left( \frac{\zz}{(x+y) \zz} \right)^2$.
	Let $A'=<(-2,1)>$ be the set of elements (additively) generated by $(-2,1)$ in $\left( \frac{\zz}{(x+y) \zz} \right)^2$. So $A'$ contains exactly the pairs $(c_1,c_2)$ such that $c_1+2c_2 \equiv 0 \pmod{x+y}.$
	Since $\gcd{2x,x+y}=1$, we note that $A'$ is a transversal and that it also can be written as $<(-2x,x)>.$
	Now let $B'=(0,x)+A'$ and $C'=(-x,0)+B'=(x,0)+A'$.
	Replacing every square here by a $2d \times 2d$ square, where the sets $A,B,C$ are formed by taking the elements on $d$ $2\times 2$ blocks on the diagonals in the corresponding squares of $A',B'$ and $C'$.
	So now we have constructed subsets $A,B,C$ in $\left( \frac{\zz}{2d(x+y) \zz} \right)^2$, where $d(x+y)=k+\ell+2.$ We can consider $A,B,C$ as subsets of $\zz^2$ as well, if we consider the coordinates modulo $2(k+\ell+2).$
	By applying Lemma~\ref{lemma4gen} on these sets $A,B,C$, we conclude $k(2)\ell$ tiles $\zz^3$. 
	For this, note that $\zz^2 \setminus (A \cup B)$ and $  \zz^2 \setminus (A \cup C)$ 
	both can be tiled with copies of $T$ which are put horizontally, while $\zz^2 \setminus (B \cup C)$ can be tiled with vertical copies of $T$.
	
	As an example, the idea is illustrated in Figure~\ref{fig:k2l} for $k=8, \ell=5$, so with $d=3,x=3$ and $y=2.$ The constructions of $A',B'$ and $C'$ in $\left( \frac{\zz}{5 \zz} \right)^2$ is presented on the left, the sets $A,B,C$ in $\left( \frac{\zz}{30 \zz} \right)^2$ are depicted in the middle and part of the tiling of $\zz^2 \setminus (B \cup C)$ is shown on the right side.
\end{proof}

\begin{figure}[h]

	\begin{tikzpicture}[scale=0.9]
	
	\clip (-0.5, -0.75) rectangle (5.55,5.55);

\foreach \t in {0,1,2} { 
	\draw[fill=black!70!white](\t,2*\t) rectangle (\t+1,2*\t+1);}

\foreach \t in {0,1} { 
	\draw[fill=black!70!white](\t+3,2*\t+1) rectangle (\t+4,2*\t+2);}

\foreach \t in {0,1,2} { 
	\draw[fill=black!40!white](\t+1,2*\t) rectangle (\t+2,2*\t+1);}

\draw[fill=black!40!white](4,1) rectangle (5,2);
\draw[fill=black!40!white](0,3) rectangle (1,4);

\foreach \t in {0,1} { 
	\draw[fill=black!10!white](\t+3,2*\t) rectangle (\t+4,2*\t+1);}

\draw[fill=black!10!white](0,4) rectangle (1,5);
\foreach \t in {0,1} { 
	\draw[fill=black!10!white](\t+1,2*\t+1) rectangle (\t+2,2*\t+2);}

	\draw[fill=black!10!white] 			(1.5,-0.7) rectangle (2,-0.2);
	\draw[fill=black!40!white] 			(3,-0.7) rectangle (3.5,-0.2);
	\draw[fill=black!70!white] 			(4.5,-0.7) rectangle (5,-0.2);
	
	\coordinate [label=center:legend:] (A) at (0.5,-0.5); 
	\coordinate [label=center:$A'$] (A) at (2.25,-0.5); 
	\coordinate [label=center:$B'$] (A) at (3.75,-0.5); 
	\coordinate [label=center:$C'$] (A) at (5.25,-0.5);

	\end{tikzpicture}
	\quad
	\begin{tikzpicture}[scale=0.9]		
	
	\clip (-0.15, -0.75) rectangle (7.55,7.55);

	\foreach \t in {0,1,2} { 
		\foreach \a in {0,0.5,1} { 
		\draw[fill=black!70!white](1.5*\t+\a,3*\t+\a) rectangle (1.5*\t+\a+0.5,3*\t+\a+0.5);}
	}
	
	\foreach \t in {0,1} { 
		\foreach \a in {0,0.5,1} { 
		\draw[fill=black!70!white](1.5*\t+4.5+\a,3*\t+1.5+\a) rectangle (1.5*\t+5+\a,3*\t+2+\a);}
	}
	
	\foreach \t in {0,1,2} { 
		\foreach \a in {0,0.5,1} { 
		\draw[fill=black!40!white](1.5*\t+1.5+\a,3*\t+\a) rectangle (1.5*\t+2+\a,3*\t+\a+0.5);}
}

	\foreach \a in {0,0.5,1} { 
	\draw[fill=black!40!white](6+\a,1.5+\a) rectangle (\a+6.5,2+\a);}

	\foreach \a in {0,0.5,1} { 
	\draw[fill=black!40!white](\a,4.5+\a) rectangle (\a+0.5,\a+5);}

	\foreach \t in {0,1} { 
		\foreach \a in {0,0.5,1} { 
		\draw[fill=black!10!white](3/2*\t+9/2+\a,3*\t+\a) rectangle (3/2*\t+5+\a,3*\t+\a+0.5);}
	}

	\foreach \a in {0,0.5,1} { 
		\draw[fill=black!10!white](\a,6+\a) rectangle (\a+0.5,6+\a+0.5);}

	\foreach \t in {0,1} {
		\foreach \a in {0,0.5,1} { 
	\draw[fill=black!10!white](\t*3/2+3/2+\a,\t*3+3/2+\a) rectangle (\t*3/2+3/2+\a+0.5,\t*3+3/2+\a+0.5);}
	}


	\draw[fill=black!10!white] 			(1.5,-0.7) rectangle (2,-0.2);
	
	\draw[fill=black!40!white] 			(3.5,-0.7) rectangle (4,-0.2);
	
	\draw[fill=black!70!white] 			(5.5,-0.7) rectangle (6,-0.2);
	
	\coordinate [label=center:legend:] (A) at (0.5,-0.5); 
	\coordinate [label=center:$A$] (A) at (2.25,-0.5); 
	\coordinate [label=center:$B$] (A) at (4.25,-0.5); 
	\coordinate [label=center:$C$] (A) at (6.25,-0.5);
	
		\draw[step=0.25,black,thin] (0,0) grid (7.5,7.5);

	\end{tikzpicture}
	\quad 
	\begin{tikzpicture}[scale=0.9]		
	
	\clip (-0.75, -0.75) rectangle (1.1,7.55);

	\foreach \a in {0,0.5} { 
		\draw[fill=black!70!white](\a,\a) rectangle (\a+0.5,\a+0.5);}

	\foreach \a in {0,0.5} { 
		\draw[fill=black!40!white](\a,4.5+\a) rectangle (\a+0.5,\a+5);}

	\draw[fill=blue](0,0.5) rectangle (0.25,2.5);
	\draw[fill=blue](0,6.25) rectangle (0.25,7.5);
	
	\draw[fill=red](0,2.5) rectangle (0.25,4.5);	
	\draw[fill=red](0,5) rectangle (0.25,6.25);
	
	\draw[fill=green](0.5,0.5) rectangle (0.25,2.5);
	\draw[fill=green](0.5,6.25) rectangle (0.25,7.5);
	
	\draw[fill=yellow](0.5,2.5) rectangle (0.25,4.5);	
	\draw[fill=yellow](0.5,5) rectangle (0.25,6.25);
	
	\draw[fill=blue](0.5,0) rectangle (0.75,0.5);
	\draw[fill=blue](0.5,1) rectangle (0.75,3);
	\draw[fill=blue](0.5,6.75) rectangle (0.75,7.5);
	
	\draw[fill=red](0.5,3) rectangle (0.75,5);	
	\draw[fill=red](0.5,5.5) rectangle (0.75,6.75);
	
	\draw[fill=green](1,0) rectangle (0.75,0.5);
	\draw[fill=green](1,1) rectangle (0.75,3);
	\draw[fill=green](1,6.75) rectangle (0.75,7.5);
	
	\draw[fill=yellow](0.75,3) rectangle (1,5);	
	\draw[fill=yellow](0.75,5.5) rectangle (1,6.75);
	
	\draw[step=0.25,black,thin] (0,0) grid (1,7.5);
	
	\end{tikzpicture}

	\caption{Construction of $A,B,C$ for $T=8(2)5$}
	\label{fig:k2l}
	
\end{figure}

\end{document}